\documentclass[a4paper,11pt]{article}
\usepackage{amsmath, amsfonts, amscd, amssymb, amsthm, enumerate, xypic}

\def\Ortho{\mathrm{O}}

\newcommand{\Irr}{\operatorname{Irr}}

\newcommand{\id}{\operatorname{id}}
\newcommand{\GL}{\operatorname{GL}}

\newcommand{\Ker}{\operatorname{Ker}}
\newcommand{\End}{\operatorname{End}}
\newcommand{\Isom}{\operatorname{Isom}}

\newcommand{\im}{\operatorname{Im}}

\newcommand{\Sp}{\operatorname{Sp}}

\renewcommand{\setminus}{\smallsetminus}

\def\F{\mathbb{F}}
\def\K{\mathbb{K}}

\def\C{\mathbb{C}}

\def\Z{\mathbb{Z}}

\renewcommand{\L}{\mathbb{L}}

\def\lcro{\mathopen{[\![}}
\def\rcro{\mathclose{]\!]}}

\theoremstyle{definition}
\newtheorem{Def}{Definition}[section]

\theoremstyle{plain}
\newtheorem{theo}{Theorem}[section]
\newtheorem{prop}[theo]{Proposition}

\theoremstyle{plain}

\theoremstyle{remark}
\newtheorem{Rems}{Remarks}
\newtheorem{Rem}[Rems]{Remark}

\title{Products of two involutions in orthogonal and symplectic groups}
\author{Cl\'ement de Seguins Pazzis\footnote{Universit\'e de Versailles Saint-Quentin-en-Yvelines, Laboratoire de Math\'ematiques
de Versailles, 45 avenue des Etats-Unis, 78035 Versailles cedex, France}
\footnote{e-mail address: dsp.prof@gmail.com}}

\begin{document}

\thispagestyle{plain}

\maketitle

\begin{abstract}
An element of a group is called bireflectional when it is the product of two involutions of the group (i.e.\ elements of order $1$ or $2$).
If an element is bireflectional then it is conjugated to its inverse.

It is known that all elements of orthogonal groups of quadratic forms are bireflectional (\cite{Wonenburger} for fields of characteristic not $2$,
\cite{Gow} for fields of characteristic $2$). F. B\"unger \cite{Bungerthesis} has characterized the elements of unitary groups (over fields of characteristic not $2$)
that are bireflectional. Yet in symplectic groups over fields with characteristic different from $2$, in general there are elements that are conjugated to their
inverse but are not bireflectional (however, over fields of characteristic $2$, every element of a symplectic group is
bireflectional, see \cite{Gow}).

In this article, we characterize the bireflectional elements of symplectic groups in terms of Wall invariants, over fields of characteristic not $2$: the result is cited without proof in B\"unger's PhD thesis, and attributed to Klaus Nielsen. We also take advantage of our approach to give a simplified proof of Wonenburger's corresponding result for orthogonal groups and general linear groups.
\end{abstract}




\section{Introduction}

\subsection{The problem}

\begin{Def}
Let $G$ be a group. An involution of $G$ is an element $x$ such that $x^2=1_G$.
An element of $G$ is called \textbf{bireflectional} when it is the product of two involutions of $G$.
\end{Def}

A basic remark is that every bireflectional element is conjugated to its inverse: indeed, if $g=ab$ where $a^2=b^2=1_G$, then
$g^{-1}=b^{-1}a^{-1}=ba=bg b^{-1}$.

Let $\F$ be a field of characteristic different from $2$, whose group of units we denote by $\F^*$.
In this article, all the vector spaces under consideration are assumed finite-dimensional.

Let $V$ be such a vector space.
Throughout, we consider a bilinear form $b : V \times V \rightarrow \F$, either symmetric
($\forall (x,y)\in V^2, \; b(x,y)=b(y,x)$) or alternating (i.e. $\forall x \in V, \; b(x,x)=0$).
Note that $b$ is alternating if and only if it is skew-symmetric, that is $\forall (x,y)\in V^2, \; b(x,y)=-b(y,x)$
(beware that this equivalence fails over fields of characteristic $2$, but we discard such fields in the present article).

We also assume that $b$ is non-degenerate, which means that $0_V$ is the only vector $x$ such that $b(x,-)=0$.
We say that $b$ is symplectic when it is non-degenerate and alternating.
If $b$ is non-degenerate, then for all $u \in \End(V)$ there is a unique $u^\star \in \End(V)$, called the $b$-adjoint of $u$, such that
$$\forall (x,y)\in V^2, \; b(u^\star(x),y)=b(x,u(y)).$$
Here, we focus on the group of isometries
$$\Isom(b):=\{u \in \GL(V) : \; \forall (x,y)\in V^2, \; b(u(x),u(y))=b(x,y)\}$$
of the bilinear form $b$ (i.e.\ the group of all automorphisms $u$ such that $u^\star=u^{-1}$).
This group is also denoted by $\Ortho(b)$ (the orthogonal group of $b$) when $b$ is symmetric, and
by $\Sp(b)$ (the symplectic group of $b$) when $b$ is alternating.

\vskip 3mm
In \cite{Wonenburger}, Wonenburger proved that if $b$ is symmetric and non-degenerate
then every element of $\Ortho(b)$ is bireflectional (her article is also widely known for proving that, in a general linear
group over a field of characteristic different from $2$, the bireflectional elements are the elements that are conjugated to their inverse).
She also noted that this result failed in symplectic groups over fields of characteristic different from $2$ (even discounting the obvious example of $2$-dimensional spaces, in which the sole symplectic involutions are $\pm \id_V$). However, Gow \cite{Gow} showed that, over fields of characteristic $2$,
every element of a symplectic group is bireflectional, and so is every element of the orthogonal group of a regular quadratic form.
Finally, in his PhD thesis \cite{Bungerthesis}, B\"unger characterized the bireflectional elements of unitary groups over fields with characteristic not $2$
(see also \cite{Bungerproc}).

Let us come back to the special case of the symplectic group $\Sp(b)$ when $b$ is symplectic (remember that $\F$ does not have characteristic $2$).
In \cite{dLC}, de La Cruz exhibited, over the field of complex numbers, a symplectic group in which there exists an element that cannot be decomposed as
the product of three involutions. Moreover he gave, over the field of complex numbers, a full characterization of the elements of $\Sp(b)$
that are bireflectional (in terms of their rational canonical form). Awa and him recently showed \cite{Awa} that the characterization
essentially fails over the field of real numbers, in the sense that there exists a real symplectic matrix which is the product
of two complex involutory symplectic matrices but fails to be the product of two real involutory symplectic matrices.
Actually, the full characterization of bireflectional elements of $\Sp(b)$ was already announced by B\"unger \cite{Bungerthesis} in his PhD thesis more
than twenty years ago, citing a mysterious article of K. Nielsen (see reference 56 therein). However, it turns out that Nielsen's article was never published,
which lead us to writing the present article in order to fill the gap.
As the techniques are fairly similar and there is recently a renewed interest in such questions
(see \cite{Awa,dLC,dSPsumprod,dSPsumexceptional,dSPsum3}), we will use the opportunity to give a clear account of the proofs of Wonenburger's result
(B\"unger's result on unitary groups could be proved with a similar method but the proof is too long to be included in this article).
In a forthcoming article, we will study the corresponding situation where
involutions are replaced with unipotent automorphisms of index $2$ (i.e.\ automorphisms $u$ such that $(u-\id_V)^2=0$).

Most of our results (and methods) involve Wall's classification of the conjugacy classes in the orthogonal and symplectic groups
\cite{Wall}: the fundamental invariants are recalled in the next three sections.

\subsection{The viewpoint of pairs}

It appears that a more efficient viewpoint on our problem is to consider pairs consisting of a form and of an isometry for this form.
More precisely:

\begin{Def}
A \textbf{$1$-isopair} $(b,u)$ consists of a non-degenerate symmetric bilinear form $b$ over a finite-dimensional vector space $V$
over $\F$, and of an isometry $u \in \Ortho(b)$.

A \textbf{$-1$-isopair} $(b,u)$ consists of a symplectic bilinear form $b$ over a finite-dimensional vector space $V$
over $\F$, and of an isometry $u \in \Sp(b)$.
\end{Def}

Let $\varepsilon \in \{-1,1\}$. Two $\varepsilon$-isopairs $(b,u)$ and $(c,v)$, with underlying vector spaces $U$ and $V$, are called \textbf{isometric} whenever there exists a vector space isomorphism $\varphi : U \overset{\simeq}{\longrightarrow} V$ such that
$$\forall (x,y)\in U^2, \; c(\varphi(x),\varphi(y))=b(x,y) \quad \text{and} \quad \varphi \circ u \circ \varphi^{-1}=v.$$
This defines an equivalence relation on the collection of all $\varepsilon$-isopairs over $\F$.

Next, the orthogonal sum of two $\varepsilon$-isopairs $(b,u)$ and $(c,v)$, with underlying vector spaces $U$ and $V$,
is defined as the pair $(b,u) \bot (b',u'):=(b \bot b', u \oplus u')$, so that
$$\forall (x,y)\in U \times V, \; \forall (x',y')\in U \times V, \quad
(b \bot b')((x,y),(x',y'))=b(x,x')+b'(y,y')$$
and
$$\forall (x,y)\in U \times V, \; (u \oplus u')(x,y)=(u(x),u'(y)).$$
One checks that orthogonal sums are compatible with isometries (i.e. replacing one summand with an isometric summand yields isometric sums).

Finally, let $(b,u)$ be an $\varepsilon$-isopair, with underlying vector space $U$, and $V$ be a linear subspace of $U$
that is $b$-regular (i.e.\ $V \cap V^{\bot_b}=\{0\}$), so that $b$ induces a non-degenerate bilinear form $b_V$ on $V$, and
$V$ is also stable under $u$, yielding an induced endomorphism $u_V$ of $V$.
Then, $(b,u)^V:=(b_V,u_V)$ is an $\varepsilon$-isopair, and so is $(b,u)^{V^{\bot_b}}$.

\begin{Def}
An $\varepsilon$-isopair $(b,u)$ is called \textbf{bireflectional} when $u$ is bireflectional in the isometry group of $b$,
i.e.\ there exist involutions $s_1$ and $s_2$ of the underlying vector space of $(b,u)$ such that $u=s_1s_2$, and $(b,s_1)$ and $(b,s_2)$ are
$\varepsilon$-isopairs.
\end{Def}

\begin{Rem}
If an $\varepsilon$-isopair is bireflectional, then so is any $\varepsilon$-isopair that is isometric to it.
\end{Rem}

\begin{Rem}
Let $(b,u)$ and $(b',u')$ be bireflectional $\varepsilon$-isopairs.
Let us choose involutions $s_1,s_2$ in the isometry group of $b$, and involutions $s'_1,s'_2$ in the isometry group of $b'$, such that
$u=s_1s_2$ and $u'=s'_1s'_2$. Then, with $S_1:=s_1 \oplus s'_1$ and $S_2:=s_2 \oplus s'_2$, one sees that $u \oplus u'=S_1S_2$ and that
$S_1$ and $S_2$ are involutions in the isometry group of $b \bot b'$.
Hence, the $\varepsilon$-isopair $(b,u) \bot (b',u')$ is bireflectional.
\end{Rem}

The converse fails. In theory, it is possible to have $(b,u) \bot (b',u')$ bireflectional
without having $(b,u)$ or $(b',u')$ bireflectional.

\begin{Def}
An $\varepsilon$-isopair is called \textbf{indecomposable} when it is nontrivial (i.e.\ defined on a nonzero
vector space) and it is not isometric to the orthogonal direct sum of two non-trivial isopairs
(in other words, such a pair $(b,u)$, over a space $V$, is indecomposable when it is nonzero and there is no $b$-orthogonal decomposition
$V=V_1 \overset{\bot_b}{\oplus} V_2$ in which $V_1$ and $V_2$ are nonzero and stable under $u$).
\end{Def}

\subsection{A review of conjugacy classes and extensions of bilinear forms}\label{section:formextension}

We denote by $\Irr(\F)$ the set of all irreducible polynomials $p \in \F[t]$ such that $p \neq t$ (beware that
in \cite{dSPsum2quad} this notation did not exclude the polynomial $t$).

Let $V$ be a vector space over $\F$.
The classification of conjugacy classes in the algebra $\End(V)$ is well known: here, we will use the viewpoint of the primary invariants.
Every endomorphism $u$ of $V$ is the direct sum of cyclic endomorphisms whose minimal polynomials are powers of
monic irreducible polynomials: those minimal polynomials are uniquely determined up to permutation, and they are called the
primary invariants of $u$. For $p \in \Irr(\F) \cup \{t\}$ and $r \geq 1$, we denote by $n_{p,r}(u)$ the number of summands with minimal polynomial
$p^r$ in such a decomposition (it is the \textbf{Jordan number} of $u$ attached to the pair $(p,r)$).
The similarity class of $u$ is then determined by the data of the family $(n_{p,r}(u))_{p \in \Irr(\F) \cup \{t\}, r \geq 1}$.

Now, let $u$ be a $b$-isometry.
For $u$ to be bireflectional in $\Isom(b)$, it is necessary that $u$ be similar to $u^{-1}$. However, this is automatic for a $b$-isometry because
classically $u^\star$ is similar to $u$!

Given a monic polynomial $p(t)$ of degree $d$ such that $p(0) \neq 0$, we denote by
$$p^\sharp(t):=p(0)^{-1}t^d p(1/t)$$
the reciprocal polynomial of $p$ (note that $(p^\sharp)^\sharp=p$), and we say that $p$ is a \textbf{palindromial}
whenever $p^\sharp=p$.
For an irreducible palindromial $p \in \Irr(\F)$, either $p=t\pm 1$ or the degree of $p$ is even and $p(0)=1$.
An automorphism $u$ of a vector space $V$ is similar to its inverse if and only if, for all
$p \in \Irr(\F)$ such that $p^\sharp \neq p$, and for all $r \geq 1$, one has $n_{p,r}(u)=n_{p^\sharp,r}(u)$.

Now, let $p \in \Irr(\F)$ be a palindromial that is different from $t+ 1$ and $t-1$. Denoting by $2d$ the degree of $p$, it is known that
$p(t)=t^d R(t+t^{-1})$ for a unique $R \in \Irr(\F)$ of degree $d$.
Consider the field $\L:=\F[t]/(p)$. The classes of $t$ and its inverse are roots of $p$ in $\L$, and hence
there is a unique field automorphism $x \mapsto x^\bullet$ of $\L$ that is an automorphism of $\F$-algebra and that takes
the class of $t$ to its inverse. This allows use to see $\L$ as a field with a non-identity involution, and we shall consider Hermitian forms
over $\L$ with respect to this involution in the remainder of the article. Moreover,
the subfield $\K:=\{\lambda \in \L : \; \lambda^\bullet=\lambda\}$ is isomorphic to $\F[t]/(R)$
through the isomorphism $\psi$ that takes the class $\overline{t}+\overline{t}^{-1}$ to the class of $t$ modulo $R$.
We consider the $\F$-linear form $e_R : \F[t]/(R) \rightarrow \F$ that takes the class of $1$ to $1$, and the class of $t^k$ to $0$ for all
$k \in \lcro 1,d-1\rcro$,
and finally we set
$$f_p : \lambda \in \L \mapsto e_R(\psi(\lambda+\lambda^\bullet)) \in \F.$$
Hence, $f_p$ is a nonzero linear form on the $\F$-vector space $\L$, and it is invariant under the involution $\lambda \mapsto \lambda^\bullet$.

The next step is a general construction: let $V$ be an $\L$-vector space, and
$B : V^2 \rightarrow \F$ be an $\F$-bilinear form such that
$\forall (x,y)\in V^2, \; \forall \lambda \in \mathbb{L}, \; B(\lambda^\bullet x,y)=B(x,\lambda y)$.
For all $(x,y)\in \L^2$, the mapping $\lambda \mapsto B(x,\lambda y)$ is an $\F$-linear form on
$V$ and hence it reads $\lambda \mapsto f_p(\lambda B^\L(x,y))$ for a unique $B^\L(x,y)\in \L$.
This yields a mapping $B^\L : V^2 \rightarrow \L$ and one checks that it is $\F$-bilinear and even $\L$-right-linear.
Moreover, if $B$ is non-degenerate then so is $B^\L$.
Now, assume that $B$ is non-degenerate and symmetric. Let $(x,y) \in V^2$. Then, for all $\lambda \in \F$,
$$B(y,\lambda x)=B(\lambda^\bullet y,x)=B(x,\lambda^\bullet y)=f_p\bigl(\lambda^\bullet B^\L(x,y)\bigr)=f_p\bigl(\lambda B^\L(x,y)^\bullet\bigr)$$
and hence $B^{\mathbb{L}}(y,x)=B^\L(x,y)^\bullet$. Likewise, if $B$ is skew-symmetric and non-degenerate one finds
$B^{\mathbb{L}}(y,x)=-B^\L(x,y)^\bullet$. To sum up:
\begin{itemize}
\item If $B$ is symmetric and non-degenerate, then $B^{\mathbb{L}}$ is a non-degenerate Hermitian form on the $\L$-vector space $V$;
\item If $B$ is symplectic then $B^{\mathbb{L}}$ is a non-degenerate skew-Hermitian form on the $\L$-vector space $V$.
\end{itemize}
Remember finally that skew-Hermitian forms are in one-to-one correspondence with Hermitian forms: by taking an element
$\eta\in \L \setminus \{0\}$ such that $\eta^\bullet =-\eta$ (such an element always exists, e.g.\ $\overline{t}-\overline{t}^{-1}$), the Hermitian form $h$ gives rise to the
skew-Hermitian form $\eta h$ and vice versa.

\subsection{A review of conjugacy classes in orthogonal and symplectic groups}

Let now $(b,u)$ be an $\varepsilon$-isopair, and let $r \geq 1$ be a positive integer.
Let $p \in \Irr(\F) \setminus \{t+1,t-1\}$ be an irreducible palindromial of degree $2d$, and write $p(t)=t^d m(t+t^{-1})$ with $m \in \F[t]$ monic and
irreducible.
Set $v=u+u^{-1}=u+u^\star$, which is $b$-selfadjoint.

We denote by $V_{p,r}$ the cokernel of the (injective) linear map
$$\Ker m(v)^{r+1}/\Ker m(v)^r \longrightarrow \Ker m(v)^{r}/\Ker m(v)^{r-1}$$
induced by $m(v)$. This cokernel is naturally identified with the quotient space
$\Ker m(v)^r/(\Ker m(v)^{r-1}+(\im m(v) \cap \Ker m(v)^r))$.
We consider the bilinear form induced by
$$b_{p,r} : (x,y) \mapsto b(x,m(v)^{r-1}[y])$$
on $\Ker m(v)^{r}$.
Noting that $m(v)$ is $b$-selfadjoint, we obtain that $b_{p,r}$ is symmetric if $\varepsilon=1$, and
skew-symmetric if $\varepsilon=-1$.
The radical of $b_{p,r}$ is the intersection of $\Ker m(v)^r$
with the inverse image of $(\Ker m(v)^r)^{\bot_b}=\im m(v)^r$ under $m(v)^{r-1}$, and one easily
checks that this inverse image equals $\Ker m(v)^{r-1}+(\Ker m(v)^r \cap \im m(v))$.
Hence, $b_{p,r}$ induces a non-degenerate $\F$-bilinear form $\overline{b_{p,r}}$ on $V_{p,r}$,
a form that is symmetric if $\varepsilon=1$, skew-symmetric if $\varepsilon=-1$.

The quotient spaces $\Ker m(v)^{r+1}/\Ker m(v)^r$ and $\Ker m(v)^{r}/\Ker m(v)^{r-1}$ are naturally seen as vector spaces over $\L$, and so is
the cokernel $V_{p,r}$. Since $u$ is a $b$-isometry it turns out that
$\overline{b_{p,r}}(x,\lambda y)=\overline{b_{p,r}}(\lambda^\bullet x,y)$ for all
$\lambda \in \L$ and all $(x,y)\in (V_{p,r})^2$.
Hence, by using the extension process described in Section \ref{section:formextension}, we recover a non-degenerate form
$\overline{b_{p,r}}^\L$ on the $\L$-vector space $V_{p,r}$, either Hermitian if $b$ is symmetric, or skew-Hermitian if $b$ is symplectic.
In any case, we denote by $(b,u)_{p,r}$ this form and call it the \textbf{Hermitian Wall invariant} of $(b,u)$ attached to $(p,r)$.

The following facts are easily checked:
\begin{itemize}
\item Given an $\varepsilon$-isopair $(b',u')$ that is equivalent to $(b,u)$, the Hermitian or skew-Hermitian forms
$(b,u)_{p,r}$ and $(b',u')_{p,r}$ are equivalent;
\item Given another $\varepsilon$-isopair $(b',u')$, one has
$((b,u) \bot (b',u'))_{p,r} \simeq (b,u)_{p,r} \bot (b',u')_{p,r}$.
\end{itemize}

There are two additional sets of invariants, and here we must differentiate more profoundly between symmetric and
alternating forms.
Assume first that $\varepsilon=1$, and let $r=2k+1$ be an odd positive integer.
Again, we set $v:=u+u^{-1}$ and we note that $(v-2\id_V)^k=(-1)^k (u-\id_V)^k (u^{-1}-\id_V)^k$.
Like in the above, the symmetric bilinear form
$(x,y) \mapsto b(x,(v-2\id_V)^k(y))$ induces a non-degenerate symmetric bilinear form
$(b,u)_{t-1,2k+1}$ on $\Ker (u-\id)^r/\bigl(\Ker (u-\id)^{r-1}+(\im (u-\id) \cap \Ker (u-\id)^r)\bigr)$.
Likewise,
$(x,y) \mapsto b(x,(v+2\id_V)^k(y))$ induces a non-degenerate symmetric bilinear form
$(b,u)_{t+1,2k+1}$ on $\Ker (u+\id)^r/\bigl(\Ker (u+\id)^{r-1}+(\im (u+\id) \cap \Ker (u+\id)^r)\bigr)$.
These are the \textbf{quadratic Wall invariants} of $(b,u)$.

The classification of $1$-isopairs is then given in the next theorem. It will not be necessary for the proof of Wonenburger's theorem, but
it is interesting to state it in parallel with the corresponding result on symplectic groups.

\begin{theo}[Wall's theorem for orthogonal groups, see \cite{Wall}]
Let $(b,u)$ and $(b',u')$ be two $1$-isopairs over a field $\F$ with characteristic different from $2$.
For $(b,u)$ to be isometric to $(b',u')$, it is necessary and sufficient that all the following conditions hold:
\begin{enumerate}[(i)]
\item For all $p \in \Irr(\F)$ such that $p \neq p^\sharp$, and all $r \geq 1$, one has $n_{p,r}(u)=n_{p,r}(u')$;
\item For all $p \in \Irr(\F) \setminus \{t+1,t-1\}$ such that $p=p^\sharp$, and all $r \geq 1$,
the Hermitian forms $(b,u)_{p,r}$ and $(b',u')_{p,r}$ are equivalent;
\item For every $\eta \in \{-1,1\}$ and every odd integer $r \geq 1$, the symmetric bilinear forms
$(b,u)_{t-\eta,r}$ and $(b',u')_{t-\eta,r}$ are equivalent.
\end{enumerate}
\end{theo}

Next, we state the classification of $-1$-isopairs.
Let $(b,u)$ be such an isopair, and set $v:=u+u^{-1}$.
Let $r=2k+2$ be an even positive integer.
The bilinear form $(x,y) \mapsto b(x,(u-u^{-1})(v-2\id_V)^k(y))$ is symmetric.
Noting that $(u-u^{-1})(v-2\id_V)^k=(-1)^k u^{-(k+1)}(u+\id_V) (u-\id_V)^{2k+1}$, one finds that
this bilinear form induces a non-degenerate symmetric bilinear form on the
quotient space $\Ker (u-\id_V)^r/\bigl(\Ker (u-\id_V)^{r-1}+(\im (u-\id_V) \cap \Ker (u-\id_V)^r)\bigr)$,
and we denote this form by $(b,u)_{t-1,r}$.
Likewise, $(x,y) \mapsto b(x,(u-u^{-1})(v+2\id_V)^k(y))$
 induces a non-degenerate symmetric bilinear form on the
quotient space $\Ker (u+\id_V)^r/\bigl(\Ker (u+\id_V)^{r-1}+(\im (u+\id_V) \cap \Ker (u+\id_V)^r)\bigr)$,
and we denote this form by $(b,u)_{t+1,r}$.
These are the \textbf{quadratic Wall invariants} of $(b,u)$.

\begin{theo}[Wall's theorem for symplectic groups, see \cite{Wall}]
Let $(b,u)$ and $(b',u')$ be two $-1$-isopairs over a field $\F$ with characteristic different from $2$.
For $(b,u)$ to be isometric to $(b',u')$, it is necessary and sufficient that:
\begin{enumerate}[(i)]
\item For all $p \in \Irr(\F)$ such that $p \neq p^\sharp$, and all $r \geq 1$, one has $n_{p,r}(u)=n_{p,r}(u')$;
\item For all $p \in \Irr(\F) \setminus \{t+1,t-1\}$ such that $p=p^\sharp$, and all $r \geq 1$,
the skew-Hermitian forms $(b,u)_{p,r}$ and $(b',u')_{p,r}$ are equivalent;
\item For every $\eta \in \{-1,1\}$ and every even integer $r \geq 2$, the symmetric bilinear forms
$(b,u)_{t-\eta,r}$ and $(b',u')_{t-\eta,r}$ are equivalent.
\end{enumerate}
\end{theo}

\begin{Rem}\label{remark:conjugatedinverse}
Let $(b,u)$ be a $-1$-isopair.
One can prove that, for every palindromial $p \in \Irr(\F) \setminus \{t+1,t-1\}$ and every integer $r \geq 1$, the Hermitian Wall invariant
$(b,u^{-1})_{p,r}$ is equivalent to $-(b,u)_{p,r}$, and likewise for every even integer $r \geq 1$ and every $\eta=\pm 1$, the quadratic Wall invariant
$(b,u^{-1})_{t-\eta,r}$ is equivalent to $-(b,u)_{t-\eta,r}$.
It follows that $u$ is conjugated to $u^{-1}$ in $\Sp(b)$ if and only if each Wall invariant of $(b,u)$ is equivalent to its \emph{opposite.}

In sharp contrast, Wall's classification of $1$-isopairs yields that every element in an orthogonal group is conjugated to its inverse,
a result which is reinforced by Wonenburger's result \cite{Wonenburger} that the conjugating automorphism can be chosen among the involutions of the said orthogonal group!
\end{Rem}

We finish with a different viewpoint, one that is well suited to our problem: a description of indecomposable pairs.

\begin{theo}\label{theo:indecomposableorthogonal}
Assume that $\chi(\F) \neq 2$.
Every indecomposable $1$-isopair $(b,u)$ satisfies one of the following properties:
\begin{itemize}
\item $u$ is cyclic with minimal polynomial $p^r$ for some palindromial $p \in \Irr(\F) \setminus \{t+1,t-1\}$
and some $r \geq 1$;
\item $u$ is cyclic with minimal polynomial $(pp^\sharp)^r$ for some
$p \in \Irr(\F)$ such that $p \neq p^\sharp$, and some $r \geq 1$;
\item $u$ is cyclic with minimal polynomial $(t-\eta)^r$ for some odd integer $r$ and some $\eta =\pm 1$;
\item $u$ has exactly two primary invariants, both equal to $(t-\eta)^r$ for some even integer $r$ and some $\eta =\pm 1$.
\end{itemize}
\end{theo}

\begin{theo}\label{theo:indecomposablesymplectic}
Assume that $\chi(\F) \neq 2$.
Every indecomposable $-1$-isopair $(b,u)$ satisfies one of the following properties:
\begin{itemize}
\item $u$ is cyclic with minimal polynomial $p^r$ for some palindromial $p \in \Irr(\F) \setminus \{t+1,t-1\}$
and some $r \geq 1$;
\item $u$ is cyclic with minimal polynomial $(pp^\sharp)^r$ for some
$p \in \Irr(\F)$ such that $p \neq p^\sharp$, and some $r \geq 1$;
\item $u$ is cyclic with minimal polynomial  $(t-\eta)^r$ for some even integer $r$ and some $\eta =\pm 1$;
\item $u$ has exactly two primary invariants, both equal to $(t-\eta)^r$ for some odd integer $r$ and some $\eta = \pm 1$.
\end{itemize}
\end{theo}

\subsection{Main results}

We are now ready to characterize bireflectional elements in symplectic groups.

\begin{theo}[Nielsen's theorem]\label{theo:Nielsen}
Let $b$ be a symplectic form over a field with characteristic different from $2$.
Let $u \in \Sp(b)$. For $u$ to be the product of two involutions in $\Sp(b)$, it is necessary and sufficient that all the
(quadratic or Hermitian) Wall invariants of $(b,u)$ be hyperbolic and that all the Jordan numbers of $u$ be even.
\end{theo}

We also cite Wonenburger's result, which we will reprove thanks to the technique used to obtain the sufficiency of the conditions in Nielsen's.

\begin{theo}[Wonenburger]\label{theo:Wonenburger}
Let $b$ be a non-degenerate symmetric bilinear form over a field with characteristic different from $2$.
Let $u \in \Ortho(b)$.
Then $u$ is the product of two involutions of $\Ortho(b)$.
\end{theo}

As seen in Remark \ref{remark:conjugatedinverse}, a symplectic automorphism $u \in \Sp(b)$ is conjugated to its inverse in $\Sp(b)$
if and only if every Wall invariant of $(b,u)$ is equivalent to its opposite.
And $u$ can be conjugated to its inverse in $\Sp(b)$ while having some of its Jordan numbers odd. For example, in
$\mathrm{SL}_2(\C)$, which is the (matrix) symplectic group for every symplectic form on $\C^2$, the diagonal matrix
$\begin{pmatrix}
i & 0 \\
0 & -i
\end{pmatrix}$ is conjugated to its inverse in $\mathrm{SL}_2(\C)$ (because all its Wall invariants are trivial), but
it is not bireflectional (in fact, only $\pm I_2$ are bireflectional in $\mathrm{SL}_2(\C)$).

Moreover, even if the Jordan numbers of $u$ are all even, it is in general insufficient that $u \in \Sp(b)$ be conjugated to its inverse in $\Sp(b)$
for $u$ to be bireflectional. Indeed, in general a quadratic/Hermitian form can be equivalent to its opposite without being hyperbolic.
For example, if $\F$ is a finite field in which $-1$ is a square, then every quadratic form over a finite extension of $\F$
is equivalent to its opposite, whereas there exist non-hyperbolic quadratic forms of even dimension.

In contrast, over the field of real numbers or the field of complex numbers, it is sufficient that $u$ have all its Jordan numbers even and be conjugated to its inverse
in $\Sp(b)$ for it to be bireflectional. Over the reals, this comes from the fact that every non-degenerate real quadratic form that is equivalent to its opposite is hyperbolic,
and ditto for Hermitian forms over the complex numbers; over the complex numbers, this comes from the fact that every complex non-degenerate quadratic form of even rank is hyperbolic. Note finally that, over an algebraically closed field of characteristic other than $2$, every symplectic transformation is conjugated to its inverse.

We finish by citing the corresponding result for general linear groups. We include a short proof of it because the proof idea will be reused to reprove
Theorem \ref{theo:Wonenburger}.

\begin{theo}[Wonenburger \cite{Wonenburger}, Djokovi\'c \cite{Djokovic} and Hoffman and Paige \cite{HoffmanPaige}]\label{theo:GL}
Let $V$ be a finite-dimensional vector space. An element of $\GL(V)$
is bireflectional if and only if it is similar to its inverse.
\end{theo}

\begin{proof}
The direct implication is already known (in any group).
So, let $u \in \GL(V)$ be similar to its inverse. Let us write its invariant factors $p_1,\dots,p_r$.
Then the ones of $u^{-1}$ are $p_1^\sharp,\dots,p_r^\sharp$. Hence $p_i=p_i^\sharp$ for all $i \in \lcro 1,r\rcro$,
and we gather that $u$ is a direct sum of cyclic automorphisms, each of which has its minimal polynomial a palindromial.
It will suffice to prove that each summand is bireflectional (in the corresponding general linear group).

So, we are left with proving the result in the special case where $u$ is cyclic and its minimal polynomial $p$ is a palindromial.
Let us then consider a cyclic vector $x$ for $u$. We denote by $R=\F[t,t^{-1}]$ the $\F$-algebra of all Laurent polynomials with coefficients in $\F$.
The $\F[t]$-module structure on $V$ induced by $u$ is isomorphic to the quotient of $R$ by the ideal generated by $p$
(through an isomorphism that takes the class of $t^k$ to $u^k(x)$ for all $k \in \Z$).
We note that the natural involution of the $\F$-algebra $R$ induced by $t \mapsto t^{-1}$
maps $p$ to $p(t^{-1})=p(0) t^{-d} p^\sharp$ where $d:=\deg(p)$, and hence it leaves the ideal generated by $p$ invariant.
This yields an involution $s : V \rightarrow V$ of the $\F$-vector space $V$ that takes $u^k(x)$ to $u^{-k}(x)$ for all $k \in \Z$.
Then we note that $us$ takes $u^k(x)$ to $u^{1-k}(x)$ for all $k \in \Z$, and it follows that $us$ is an involution of the $\F$-vector space $V$.
Hence $u=(us)s$ is bireflectional in $\GL(V)$.
\end{proof}

\subsection{Structure of the article}

The remainder of the article is structured as follows. In Section \ref{section:extension}, we develop the main tool to study
bireflectional elements in orthogonal and symplectic groups: the orthogonal or symplectic extension of an automorphism of a vector space.
Most authors tend to frame this construction in terms of block matrices, but we prefer to explain it through duality theory as in \cite{dSPsum2quad}.
Using these tools, Theorem \ref{theo:Nielsen} is proved in Section \ref{section:Nielsen}, whereas
Theorem \ref{theo:Wonenburger} is easily proved in Section \ref{section:Wonenburger} thanks to the description of indecomposable $1$-isopairs
from Theorem \ref{theo:indecomposableorthogonal}.

\section{Orthogonal and symplectic extensions of an automorphism}\label{section:extension}

Throughout, we let $V$ be  a finite-dimensional vector space over $\F$, with $\chi(\F) \neq 2$. We denote by $V^\star$ its
dual space, i.e.\ the space of all linear forms on $V$.
Given an endomorphism $u$ of $V$, we denote by $u^t$ the corresponding transposed endomorphism of $V^\star$, defined as follows:
$$\forall \varphi \in V^\star, \; u^t(\varphi)=\varphi \circ u.$$

\subsection{Definition}\label{section:extensionintroduction}

Let $\varepsilon \in \{-1,1\}$.
On the product space $V \times V^\star$, we consider the bilinear form
$$H_V^\varepsilon : \begin{cases}
(V \times V^\star)^2 & \longrightarrow \F \\
((x,\varphi),(y,\psi)) & \longmapsto \varphi(y)+\varepsilon \psi(x).
\end{cases}$$
Note that $H_V^1$ is a non-degenerate hyperbolic symmetric bilinear form, whereas $H_V^{-1}$ is a symplectic form.

Let $u \in \GL(V)$.
We consider the endomorphism
$$h(u) : \begin{cases}
V \times V^\star & \longrightarrow V \times V^\star \\
(x,\varphi) & \longmapsto \bigl(u(x), (u^{-1})^t(\varphi)\bigr),
\end{cases}$$
that is $h(u)=u \oplus (u^{-1})^t$.
Note that for all $(x,\varphi)$ and $(y,\psi)$ in $V \times V^\star$,
\begin{align*}
H_V^\varepsilon(h(u)[x,\varphi],h(u)[y,\psi]) & =(u^{-1})^t(\varphi)[u(y)]+\varepsilon\,(u^{-1})^t(\psi)[u(x)] \\
& =\varphi(y)+\varepsilon \psi(x) \\
& =H_V^\varepsilon\bigl((x,\varphi),(y,\psi)\bigr).
\end{align*}
Hence, $(H_V^\varepsilon,h(u))$ is an $\varepsilon$-isopair, which we denote by $H_\varepsilon(u)$.

Next, we examine the behavior of the pair $H_\varepsilon(u)$ with respect to similarity and direct sums.
Let $v \in \End(V)$ and $w \in \End(W)$ be similar endomorphisms, and let $\varphi : V \overset{\simeq}{\rightarrow} W$
be an isomorphism such that $w=\varphi \circ v \circ \varphi^{-1}$. Then for the isomorphism $\Phi:=\varphi \oplus (\varphi^{-1})^t$
from $V \times V^\star$ to $W \times W^\star$, we have, for all $(x,f)$ and $(y,g)$ in $V \times V^\star$,
\begin{align*}
H_W^\varepsilon\bigl(\Phi(x,f),\Phi(y,g)\bigr) & =H_W\bigl((\varphi(x),f \circ \varphi^{-1}),(\varphi(y),g \circ \varphi^{-1})\bigr) \\
& =(f \circ \varphi^{-1})(\varphi(y))+\varepsilon (g \circ \varphi^{-1})(\varphi(x)) \\
& =f(y)+\varepsilon g(x)=H_V^\varepsilon\bigl((x,f),(y,g)\bigr).
\end{align*}
Moreover, for all $(x,f)\in V \times V^\star$, we have
\begin{align*}
h(w)[\Phi(x,f)]
& =\bigl((\varphi \circ v \circ \varphi^{-1})[\varphi(x)], (\varphi \circ v^{-1} \circ \varphi^{-1})^t[f \circ \varphi^{-1}]\bigr) \\
& =(\varphi(v(x)),f \circ \varphi^{-1} \circ \varphi \circ v^{-1} \circ \varphi^{-1}) \\
& =(\varphi(v(x)),f \circ v^{-1} \circ \varphi^{-1}) \\
& =\Phi(h(v)[x,f]),
\end{align*}
leading to $h(w)=\Phi \circ h(v) \circ \Phi^{-1}$.
Hence, we have proved that
$$H_\varepsilon(v) \simeq H_\varepsilon(w).$$
In other words, the isometry ``class" of the pair $H_\varepsilon(v)$ depends only on $\varepsilon$ and on the similarity ``class" of $v$.

Finally, we examine the behavior of $H_\varepsilon(v)$ with respect to direct sums.
So let $u_1 \in \GL(V_1)$ and $u_2 \in \GL(V_2)$ be automorphisms.
We consider the direct sum $u:=u_1 \oplus u_2 \in \GL(V_1 \times V_2)$.
We shall prove that
$$H_\varepsilon(u_1 \oplus u_2) \simeq H_\varepsilon(u_1) \bot  H_\varepsilon(u_2).$$
To this end, we introduce the canonical injections $i_1 : V_1 \hookrightarrow V_1 \times V_2$ and $i_2 : V_2 \hookrightarrow V_1 \times V_2$,
and we consider the isomorphism
$$\Psi : \begin{cases}
(V_1 \times V_2) \times (V_1 \times V_2)^\star & \longrightarrow (V_1 \times V_1^\star) \times (V_2 \times V_2^\star) \\
((x_1,x_2),\varphi) & \longmapsto \bigl((x_1,\varphi \circ i_1),(x_2,\varphi \circ i_2)\bigr).
\end{cases}$$
Set $G:=H_{V_1}^\varepsilon \bot H_{V_2}^\varepsilon$.
For all $(x_1,x_2)$ and $(y_1,y_2)$ in $V_1 \times V_2$ and all $\varphi,\psi$ in $(V_1 \times V_2)^\star$, we see that
\begin{align*}
G\bigl(\Psi((x_1,x_2),\varphi),\Psi((y_1,y_2),\psi)\bigr)
& =\varphi(y_1,0)+\varepsilon \psi(x_1,0)+\varphi(0,y_2)+\varepsilon \psi(0,x_2) \\
& =\varphi(y_1,y_2)+\varepsilon \psi(x_1,x_2)
\end{align*}
and hence $\Psi$ is an isometry from $H_{V_1 \times V_2}^{\varepsilon}$ to $G$.
Finally, for all $(x_1,x_2) \in V_1 \times V_2$ and all $\varphi \in (V_1 \times V_2)^\star$, we have
$$(h(u_1) \oplus h(u_2))\bigl(\Psi((x_1,x_2),\varphi)\bigr)
=\bigl(\bigl(u_1(x_1),\varphi \circ i_1 \circ u_1^{-1}\bigr),\bigl(u_2(x_2),\varphi \circ i_2 \circ u_2^{-1}\bigr)\bigr)$$
whereas
$$h(u_1\oplus u_2)[(x_1,x_2),\varphi]=
\bigl((u_1(x_1),u_2(x_2)),\varphi \circ (u_1^{-1} \oplus u_2^{-1})\bigr)$$
so that
\begin{align*}
\Psi(h(u_1\oplus u_2)[(x_1,x_2),\varphi])
& =\bigl((u_1(x_1),\varphi \circ (u_1 \oplus u_2)^{-1} \circ i_1),(u_2(x_2),\varphi \circ (u_1 \oplus u_2)^{-1} \circ i_2)\bigr) \\
& =\bigl((u_1(x_1),\varphi \circ i_1 \circ u_1^{-1}),(u_2(x_2),\varphi \circ i_2 \circ u_2^{-1})\bigr).
\end{align*}
Hence, we have proved that $h(u_1) \oplus h(u_2)=\Psi \circ h(u_1\oplus u_2) \circ \Psi^{-1}$,
which yields the claimed isometry.

\begin{Rem}
The similarity ``class" of the automorphism $h(u) \in \GL(V \times V^\star)$ is easily deduced from the one of $u$ because $u^t$ is similar to $u$.
Simply, $n_{p,k}(h(u))=n_{p,k}(u) +n_{p^\sharp,k}(u)$ for all $p \in \Irr(\F)$ and all $k \geq 1$.
\end{Rem}

\begin{prop}\label{extensionbireflectional}
Let $u \in \GL(V)$ be similar to its inverse. Then
$h(u)$ is the product of two involutions in both $\Ortho(H_V^1)$ and $\Sp(H_V^{-1})$.
\end{prop}

\begin{proof}
Indeed, by Theorem \ref{theo:GL}, $u=s_1s_2$ for some involutions $s_1,s_2$ in the group $\GL(V)$. Then $h(u)=h(s_1)h(s_2)$ and
$h(s_1)$ and $h(s_2)$ are involutions, both in $\Ortho(H_V^{1})$ and in $\Sp(H_V^{-1})$.
\end{proof}

\subsection{The kappa construction}\label{kappasection}

Now, let $b : V \times V \rightarrow \F$ be a non-degenerate bilinear form, and let $\varepsilon \in \{1,-1\}$.
We consider the isomorphism $L_b : x \in V \mapsto b(-,x) \in V^\star$ and we set
$$\kappa(b) : \begin{cases}
V \times V^\star & \longrightarrow V \times V^\star \\
(x,\varphi) & \longmapsto \bigl(L_b^{-1}(\varphi),L_b(x)\bigr).
\end{cases}$$
Clearly, $\kappa(b)$ is an involution that exchanges $V \times \{0\}$ and $\{0\} \times V^\star$.

Now, assume that $b$ is symmetric if $\varepsilon=1$, and alternating otherwise.
Let $(x,\varphi)$ and $(y,\psi)$ in $V \times V^\star$.
Then,
\begin{align*}
H_V^\varepsilon\bigl(\kappa(b)[x,\varphi],\kappa(b)[y,\psi]\bigr) & = L_b(x)[L_b^{-1}(\psi)]+\varepsilon\, L_b(y)[L_b^{-1}(\varphi)] \\
& = b(L_b^{-1}(\psi),x)+\varepsilon\, b(L_b^{-1}(\varphi),y) \\
& = \varepsilon\, b(x,L_b^{-1}(\psi))+b(y,L_b^{-1}(\varphi)) \\
& = \varepsilon\, L_b(L_b^{-1}(\psi))[x]+L_b(L_b^{-1}(\varphi))[y] \\
& = \varepsilon\, \psi(x)+\varphi(y) \\
& = H_V^{\varepsilon}\bigl((x,\varphi),(y,\psi)\bigr).
\end{align*}
Hence, $(H_V^\varepsilon,\kappa(b))$ is an $\varepsilon$-isopair.

Conversely, let $v$ be an involution of $V \times V^\star$ such that $(H_V^\varepsilon,v)$ is an $\varepsilon$-isopair
and $v$ exchanges $V \times \{0\}$ and $\{0\} \times V^\star$.
Thus, we have a vector space isomorphism $f : V \overset{\simeq}{\rightarrow} V^\star$ such that
$$\forall (x,\varphi) \in V \times V^\star, \; v(x,\varphi)=(f^{-1}(\varphi),f(x)).$$
Let us consider the bilinear form $b : (x,y)\in V^2 \mapsto f(y)[x]$. Then $v=\kappa(b)$.
For all $(x,y)\in V^2$, we have
\begin{multline*}
b(y,x)=f(x)[y]=H_V^\varepsilon\bigl((0,f(x)),(y,0)\bigr)
=H_V^\varepsilon\bigl(v(x,0),v(0,f(y))\bigr) \\
=H_V^\varepsilon\bigl((x,0),(0,f(y))\bigr)=\varepsilon\, f(y)[x]=\varepsilon\, b(x,y).
\end{multline*}
Hence, $b$ is symmetric if $\varepsilon=1$, and alternating otherwise.

\subsection{Splitting an extended pair}

Now, let $b$ and $c$ be two non-degenerate bilinear forms on $V$, either both symmetric or both alternating.
Setting $\eta:=1$ if $b,c$ are symmetric, and $\eta:=-1$ otherwise,
we have $L_b^t=\eta\,L_b \circ i_V^{-1}$ and $L_c^t=\eta\,L_c \circ i_V^{-1}$ where $i_V : V \overset{\simeq}{\rightarrow} V^{\star \star}$
is the canonical biduality isomorphism. Setting
$u:=L_b^{-1} L_c \in \GL(V)$, we deduce that $u^t=L_c^t (L_b^t)^{-1}=L_c L_b^{-1}$ and hence $(u^t)^{-1}=L_b L_c^{-1}$.
It follows that
$$\kappa(b) \circ \kappa(c)=h(u).$$

\begin{Rem}
Let $u \in \GL(V)$. Remember from Frobenius's theorem \cite{Frobenius} that, for all $u \in \GL(V)$, there exist
non-degenerate symmetric bilinear forms $b,c$ on $V$ such that $u=L_b^{-1} L_c$. Hence, the preceding construction shows that
 $h(u)$ is bireflectional in $\Ortho(H_V^1)$. This construction will not be used in our proof of Theorem \ref{theo:Wonenburger}, however.
\end{Rem}

\subsection{The Hermitian invariants of a hyperbolic extension}

Here, we shall prove the following result:

\begin{prop}\label{prop:extensionhyperbolic}
Let $u \in \GL(V)$. Then all the Wall invariants of $H_{-1}(u)$ are hyperbolic.
\end{prop}

\begin{proof}
Using the last part of Section \ref{section:extensionintroduction} and the fact that the Wall invariants are compatible with direct sums,
we find that it suffices to consider the case where $u$ is cyclic with minimal polynomial $p^r$ for some
$p \in \Irr(\F)$ and some $r \geq 1$.
If in addition $p \neq p^\sharp$ then $h(u)$ has exactly two primary invariants, $p^r$ and $(p^\sharp)^r$, and hence
all the quadratic and Hermitian invariants of $H_{-1}(u)$ vanish.

This leaves us with only three remaining cases.

\vskip 3mm
\noindent \textbf{Case 1: $p=(t-\eta)^r$ for some odd $r \geq 1$ and some $\eta =\pm 1$.} \\
Then $h(u)$ has exactly two Jordan cells, both attached to $t-\eta$ and of size $r$, and hence all the Wall invariants
of $H_{-1}(u)$ vanish.

\vskip 3mm
\noindent \textbf{Case 2: $p^\sharp=p$ and $p \not\in \{t+1,t-1\}$.} \\
Then $h(u)$ has exactly two Jordan cells, both attached to $(p,r)$.
The pair $P:=H_{-1}(u)$ has exactly one non-zero Wall invariant, namely $P_{p,r}$,
and it is defined on $W:=(V \times V^\star)/\im p(h(u))$, which is an $\L$-vector space of dimension $2$ for
$\L:=\F[t]/(p)$. Note also that $\im p(h(u))=\im p(u) \times \im p(u^{-1})^t$.
Set $v:=h(u)+h(u)^{-1}$.
In order to prove that the skew-Hermitian form $P_{p,r}$ is hyperbolic, it suffices to exhibit a non-zero isotropic vector for it.
Choose $x \in V \setminus \im p(u)$. Denote by $2d$ the degree of $p$, and set $m \in \F[t]$ such that $p(t)=t^d m(t+t^{-1})$.
For all $q\in \F[t]$, the second component of $q(h(u)) m(v)^{r-1}[x,0]$ is obviously zero, which yields
$$H_V^{-1} \bigl((x,0),q(h(u)) m(v)^{r-1}[x,0]\bigr)=0.$$
This leads to $\forall \lambda \in \L, \; c(X,\lambda\,X)=0$, where $X$ denotes the class of $(x,0)$ in $W$ and $c$ the $\F$-bilinear form induced by
$b : (x',y') \mapsto H_V^{-1}\bigl(x',m(v)^{r-1}(y')\bigr)$ on $W$.
This leads to $P_{p,r}(X,X)=0$. Since $X\neq 0$, we deduce that $P_{p,r}$ is isotropic, and hence hyperbolic.

\vskip 3mm
\noindent \textbf{Case 3: $p=(t-\eta)^r$ for some even $r \geq 1$ and some $\eta \geq 1$.} \\
In that case, all the Hermitian Wall invariants of $P:=H_{-1}(u)$ vanish, and so do all the quadratic Wall invariants
with the exception of $P_{t-\eta,r}$, which is a non-degenerate symmetric bilinear form on the $2$-dimensional space
$(V \times V^\star)/\im (h(u)-\eta\,\id)$. Then, with a similar proof as in Case 2, one chooses $x \in V \setminus \im (u-\eta\,\id_V)$
and one finds that the class $X$ of $(x,0)$ in $(V \times V^\star)/\im (h(u)-\eta\,\id)$ is a non-zero isotropic vector for $P_{t-\eta,r}$,
which shows that this form is hyperbolic.
\end{proof}

\section{Symplectic automorphisms that are products of two symplectic involutions}\label{section:Nielsen}

In this section, we prove Theorem \ref{theo:Nielsen}. In fact, we shall prove an even more enlightening result:

\begin{theo}
Let $(b,u)$ be a $-1$-isopair. The following conditions are equivalent:
\begin{enumerate}[(i)]
\item $u$ is the product of two involutions in $\Sp(b)$.
\item All the Wall invariants of $(b,u)$ are hyperbolic, and all the Jordan numbers of $u$ are even.
\item $(b,u)$ is isometric to $H_{-1}(v)$ for some vector space $V$ and some automorphism $v$ of $V$
that is similar to its inverse (i.e.\ bireflectional in $\GL(V)$).
\end{enumerate}
\end{theo}

The implication (iii) $\Rightarrow$ (i) is a straightforward consequence of Proposition \ref{extensionbireflectional}.
Let us now prove that (ii) implies (iii). Assume that condition (ii) holds.
Because all the Jordan numbers of $u$ are even, we can choose a vector space $V$ together with an automorphism $v$ of $V$
such that $n_{p,r}(v)=\frac{1}{2} n_{p,r}(u)$ for all $p \in \Irr(\F)$ and all $r \geq 1$.
We shall prove that $(b,u) \simeq H_{-1}(v)$.

First of all, $n_{p,r}(v)=\frac{1}{2} n_{p,r}(u)=\frac{1}{2} n_{p^\sharp,r}(u)=n_{p^\sharp,r}(v)=n_{p,r}(v^{-1})$
for all $p \in \Irr(\F)$ and all $r \geq 1$, to the effect that $v$ is conjugated to $v^{-1}$ in $\GL(V)$.
Moreover $n_{p,r}(h(v))=n_{p,r}(v)+n_{p^\sharp,r}(v)=n_{p,r}(u)$ for all $p \in \Irr(\F)$ and all $r \geq 1$.
By Wall's classification of conjugacy classes in symplectic groups, it only remains to prove that
$h(v)$ and $u$ have equivalent Wall invariants. To obtain this, note that for all $p \in \Irr(\F)$ such that $p^\sharp=p$,
both the bilinear or skew-Hermitian forms $H_{-1}(v)_{p,r}$ and $(b,u)_{p,r}$
are hyperbolic (by Proposition \ref{prop:extensionhyperbolic} for the former, by condition (ii) for the latter), and
they have the same rank $n_{p,r}(u)$, therefore they are equivalent.
We conclude that $(b,u) \simeq H_{-1}(v)$, which proves condition (iii).

\vskip 3mm
It remains to prove that condition (i) implies condition (ii).
So, assume that $u=ss'$ for some involutions $s,s'$ in $\Sp(b)$. We shall forget $s'$ and simply note that
$sus^{-1}=u^{-1}$ (since $s's=u^{-1}$).

\vskip 3mm
\noindent \textbf{The Hermitian Wall invariants}

Let $p \in \Irr(\F)$ be a palindromial of even degree $2d$, and write $p(t)=t^d m(t+t^{-1})$ for some monic $m \in \F[t]$ of degree $d$.
Let $r \geq 1$ be a non-negative integer. We shall consider the $(b,u)_{p,r}$ Wall invariant.
Set $\L:=\F[t]/(p)$ and denote by $\lambda \mapsto \lambda^\bullet$ its involution that takes the class of $t$ to its inverse.

Set $v:=u+u^{-1}$ and note that $s$ commutes with $v$.
It follows that $s$ induces an endomorphism $\overline{s}$ of
the underlying $\F$-vector space $V_{p,r}$ of $(b,u)_{p,r}$, but it is not $\L$-linear, rather it is
$\L$-semilinear because $sus^{-1}=u^{-1}$.
We shall inspect the properties of $\overline{s}$ with respect to the form $(b,u)_{p,r}$.
First of all, we recall that this form is defined thanks to the $\F$-bilinear
form $c : V_{p,r}^2 \rightarrow \F$ induced by $(x,y) \mapsto b(x,m(v)^{r-1}(y))$.
Note that $s$ is symmetric for this form. Moreover, for all $x,y$ in
$V$, we have
$$b\bigl(s(x),m(v)^{r-1}(s(y))\bigr)=
b\bigl(x,s^{-1}(m(v)^{r-1}(s(y)))\bigr)=b\bigl(x,m(v)^{r-1}(y)\bigr),$$
and it follows that $\overline{s}$ is an isometry for $c$.
Finally, let $x,y$ belong to $V_{p,r}$.
Then, for all $\lambda \in \L$,
$$c(\overline{s}(x),\lambda \overline{s}(y))=c(\overline{s}(x),\overline{s}(\lambda^\bullet y))
=c(x,\lambda^\bullet y)=f_p(\lambda^\bullet (b,u)_{p,r}(x,y))=f_p(\lambda (b,u)_{p,r}(x,y)^\bullet)$$
and hence
$$(b,u)_{p,r}(\overline{s}(x),\overline{s}(y))=(b,u)_{p,r}(x,y)^\bullet=-(b,u)_{p,r}(y,x).$$
Now, consider the subfield $\K:=\{\lambda \in \L : \lambda^\bullet=\lambda\}$: one has $[\L:\K]=2$,
and $\overline{s}$ is a $\K$-linear map.
It follows from the above identity that $(b,u)_{p,r}$ induces an alternating bilinear form on the $\K$-vector space
$W:=\Ker (\overline{s}-\id)$. We can therefore choose a $\K$-linear subspace $Z$ of $W$ such
that $\dim_{\K} Z\geq \frac{\dim_\K W}{2}$ and $(b,u)_{p,r}$ vanishes everywhere on $Z^2$.

Let us now choose $\alpha \in \L \setminus \{0\}$ such that $\alpha^\bullet=-\alpha$. Because $\overline{s}$
is $\L$-semilinear, the mapping $x \mapsto \alpha x$ induces a $\K$-linear isomorphism from $W$ to $\Ker (\overline{s}+\id)$, and hence
$Z+\alpha Z$ is an $\L$-linear subspace of $V_{p,r}$ with dimension over $\K$ at least $\frac{\dim_\K V_{p,r}}{2}$,
and hence with dimension over $\L$ at least $\frac{\dim_\L V_{p,r}}{2}\cdot$
Since $(b,u)_{p,r}$ is sesquilinear it is clear that it vanishes everywhere on $(Z+\alpha Z)^2$, and we conclude that it is
hyperbolic.

\vskip 3mm
\noindent \textbf{The quadratic Wall invariants}

Next, we consider the Wall invariant $(b,u)_{t-\eta,2r}$ for some $\eta \in \{-1,1\}$ and some $r \geq 1$.
This invariant is the (non-degenerate) symmetric bilinear form $c$ induced
by $(x,y) \mapsto b\bigl(x,(u-u^{-1})(v-2\eta\id)^{r-1}(y)\bigr)$ on the quotient space
$V_{t-\eta,2r}:=\Ker(u-\eta \id)^{2r}/\bigl(\Ker(u-\eta \id)^{2r-1}+(\im (u-\eta\id) \cap \Ker(u-\eta \id)^{2r})\bigr)$.
The endomorphism $s$ induces an endomorphism $\overline{s}$ of $V_{t-\eta,2r}$.
This time around, we note that $s$ skew-commutes with $w:=(u-u^{-1})(v-2\eta\id)^{r-1}$
and hence $c(\overline{s}(x),\overline{s}(y))=-c(x,y)$ for all $x,y$ in $V_{t-\eta,2r}$.
In particular, both subspaces $\Ker(\overline{s}-\id)$ and $\Ker(\overline{s}+\id)$
are totally $c$-isotropic, and since at least one of them has its dimension greater than or equal to
$\frac{1}{2} \dim V_{t-\eta,2r}$ we deduce that $c$ is hyperbolic (note that this shows, since $c$ is non-degenerate, that
$\Ker(\overline{s}-\id)$ and $\Ker(\overline{s}+\id)$ actually have the same dimension).

\vskip 3mm
\noindent \textbf{The Jordan numbers}

Finally, we prove that the Jordan numbers of $u$ are all even.
As the Hermitian and quadratic Wall invariants of $u$ are all hyperbolic, all the corresponding Jordan numbers are even,
and hence only one type of Jordan number of $u$ needs to be examined: those of the form
$n_{p,r}(u)$ with $p \in \Irr(\F)$ such that $p \neq p^\sharp$, and $r \geq 1$.
For any $p \in \Irr(\F)$, denote by $E_p:=\underset{k \geq 0}{\bigcup} \Ker p(u)^k$ the characteristic subspace of $u$ attached to $p$.
Since $u$ is a $b$-isometry, it is known that the $b$-orthogonal complement of $E_p$
is $\underset{q \in \Irr(\F) \setminus \{p^\sharp\}}{\bigoplus} E_q$, and in particular if $p \neq p^\sharp$ then
$E_p \oplus E_{p^\sharp}$ is $b$-regular whereas $E_p$ and $E_{p^\sharp}$ are totally $b$-singular.

Now, fix $p \in \Irr(\F)$ such that $p \neq p^\sharp$, and set $V:=E_p$, $V':=E_{p^\sharp}$ and $W:=V \oplus V'$.
Noting that $s u s^{-1}=u^{-1}$, we see that $s p(u) s^{-1}= p(u^{-1})$, and it follows that $s$ maps $V$ into $V'$ and
$V'$ back into $V$.
Since $V$ and $V'$ are totally $b$-singular whereas $W$ is $b$-regular, the mapping
$$\varphi : \begin{cases}
V' & \longrightarrow V^\star \\
x' & \longmapsto b(x',-)_{|V}
\end{cases}$$
is a vector space isomorphism, and hence
$$\Phi : \begin{cases}
W & \longrightarrow V \times V^\star \\
x+x' & \longmapsto (x,\varphi(x'))
\end{cases}$$
is also a vector space isomorphism.
Let $x,y$ in $V$ and $x',y'$ in $V'$. Then
$$H_V^{-1}(\Phi(x+x'),\Phi(y+y'))=b(x',y)-b(y',x)=
b(x',y)+b(x,y')=b(x+x',y+y'),$$
where the last inequality comes from the fact that $b$ is totally singular on both $V$ and $V'$.
Hence, $\Phi$ is an isometry from the symplectic form $\overline{b}$ induced by $b$ on $W$
to $H_V^{-1}$.
Denoting by $v$ the automorphism of $W$ induced by $u$, we deduce that the $-1$-isopair $(\overline{b},v)$
is isometric to $(H_V^{-1}, \Phi \circ v \circ \Phi^{-1})$, and
$a:=\Phi \circ s \circ \Phi^{-1}$ is an involution of $\Sp(H_V^{-1})$ that exchanges $V \times \{0\}$ and $\{0\} \times V^\star$.
As seen in Section \ref{kappasection}, this yields $a=\kappa(B)$ for some symplectic form $B$ on $V$.
Likewise, $a^{-1} \circ \Phi \circ v \circ \Phi^{-1}$ is an involution of $\Sp(H_V^{-1})$ that exchanges $V \times \{0\}$
and $\{0\} \times V^\star$, and hence it equals $\kappa(C)$ for some symplectic form $C$ on $V$.

We conclude that $\Phi \circ v \circ \Phi^{-1}=\kappa(B) \circ \kappa(C)=h(w)$ for $w:=L_B^{-1} L_C$.
As $B$ and $C$ are symplectic forms on $V$, it is known that all the Jordan numbers of $w$ are even (see e.g.\ \cite{Scharlaupairs}).
Hence, all the Jordan numbers of $v$ are even. In particular $n_{p,r}(v)$, which equals $n_{p,r}(u)$, is even.
This completes the proof of Theorem \ref{theo:Nielsen}.

\section{Bireflectionality in orthogonal groups}\label{section:Wonenburger}

In this short section, we use the opportunity of the previous constructions to give a quick proof of
Wonenburger's result on orthogonal groups (Theorem \ref{theo:Wonenburger}). Once more, it suffices to consider the case where
$(b,u)$ is an indecomposable $1$-isopair.
We have seen in Theorem \ref{theo:indecomposableorthogonal} that there are four types of such indecomposable pairs, and for the sake of the proof
we regroup them in only two cases: either $u$ is cyclic and its minimal polynomial is a palindromial, or $u$ has exactly two
invariant factors, both equal to $(t-\eta)^{2r}$ for some $\eta =\pm 1$ and some integer $r \geq 1$.

\vskip 3mm
\noindent \textbf{Case 1: $u$ has exactly two invariant factors, both equal to $(t-\eta)^{2r}$ for some $\eta =\pm 1$
and some integer $r \geq 1$.}

Choose a cyclic endomorphism $v$ with minimal polynomial $(t-\eta)^{2r}$ (and underlying vector space $W$). Then $v$ is similar to its inverse, and it follows from
Proposition \ref{extensionbireflectional} that $h(v)$ is bireflectional in $\Ortho(H_W^1)$. Besides,
$H_1(v)$ and $(b,u)$ have all their Wall invariants equal to zero, and they have the same Jordan numbers.
Hence, $H_1(v) \simeq (b,u)$, and we conclude that $u$ is bireflectional in $\Ortho(b)$.

\vskip 3mm
\noindent \textbf{Case 2: $u$ is cyclic and its minimal polynomial is a palindromial.}

Denote by $V$ the underlying vector space of $(b,u)$, and choose a cyclic vector $x$ for $u$. We have seen in the proof of Theorem \ref{theo:GL} that
there exists an involution $s \in \GL(V)$ such that $s(u^k(x))=u^{-k}(x)$ for all $k \in \Z$, and $us$ is also an involution.
Better still, $s$ is a $b$-isometry: indeed, for all $k,l$ in $\Z$, we have
$$b\bigl(s(u^k(x)),s(u^l(x))\bigr)=b\bigl(u^{-k}(x),u^{-l}(x)\bigr)=b\bigl(u^l(x),u^k(x)\bigr)=b\bigl(u^k(x),u^l(x)\bigr).$$
We deduce that $us$ is also an isometry for $b$, and finally $u=(us)s$ is bireflectional in $\Ortho(b)$.

\begin{Rem}
Professor Klaus Nielsen pointed us on a serious error in
Wonenburger's proof \cite{Wonenburger} of Theorem \ref{theo:Wonenburger}: on the last line of page 337, the vector $w$ is not in $P_1$ if $n$ is even
and $S$ is an orthogonal transformation, rather it is in $Q_1$. This essentially corresponds to the situation in Case 1 above.
\end{Rem}


\begin{thebibliography}{1}
\bibitem{Awa}
D. Awa, R.J. de La Cruz,
{Each real symplectic matrix is a product of symplectic involutions,}
Linear Algebra Appl.
{\bf 589} (2020), 85--95.

\bibitem{Bungerproc}
F. B\"unger,
{Products of involutions in unitary groups,}
Proceedings of the 4th International Congress of Geometry, Academy of Athens, Aristotle University of Thessaloniki, pp. 93-100, 1997.

\bibitem{Bungerthesis}
F. B\"unger, Involutionen als Erzeugende in unit\"aren Gruppen,
\newblock{PhD thesis,}
\newblock{Universit\"at zu Kiel,}
\newblock{1997.}

\bibitem{Djokovic}
D.\v{Z}. Djokovi\'c,
{Products of two involutions,}
Arch. Math. (Basel)
{\bf 18} (1967), 582--584.

\bibitem{Frobenius}
G. Frobenius,
{\"{U}ber die mit einer Matrix vertauschbaren Matrizen,}
Sitzungsber. Preuss. Akad. Wiss.
(1910), 3--15.

\bibitem{Gow}
R. Gow,
{Products of two involutions in classical groups of characteristic $2$,}
J. Algebra
{\bf 71} (1981), 583--591.

\bibitem{Gustafsonetal}
W.H. Gustafson, P.R. Halmos, H. Radjavi,
{Products of involutions,}
Linear Algebra Appl.
{\bf 13} (1976), 157--162.

\bibitem{HoffmanPaige}
F. Hoffman, E.C. Paige,
{Products of two involutions in the general linear group,}
Indiana Univ. Math. J.
{\bf 20} (1971), 1017--1020.

\bibitem{dLC}
R.J. de La Cruz,
{Each symplectic matrix is a product of four symplectic involutions,}
Linear Algebra Appl.
{\bf 466} (2015), 382--400.

\bibitem{Scharlaupairs}
W. Scharlau,
{Paare alternierender Formen,}
Math. Z.
{\bf 147} (1976), 13--19.

\bibitem{dSPsumprod}
C. de Seguins Pazzis,
{The sum and the product of two quadratic matrices: regular cases,}
Adv. Appl. Clifford Algebras
{\bf 32} (2022).

\bibitem{dSPsumexceptional}
C. de Seguins Pazzis,
{The sum of two quadratic matrices: Exceptional cases,}
Linear Algebra Appl.
{\bf 653} (2022), 357--394.

\bibitem{dSPsum3}
C. de Seguins Pazzis,
{A note on sums of three square-zero matrices,}
Linear Multilinear Algebra
{\bf 65} (2017), 787--805.

\bibitem{dSPsum2quad}
C. de Seguins Pazzis,
{Sums of two square-zero selfadjoint or skew-selfadjoint endomorphisms,}
preprint, arXiv, https://arxiv.org/abs/2210.03955

\bibitem{Wall}
G.E. Wall,
{On the conjugacy classes in orthogonal, symplectic and unitary groups,}
J. Austral. Math. Soc.
{\bf 3-1} (1963), 1--62.

\bibitem{Wonenburger}
M.J. Wonenburger,
{Transformations which are products of two involutions,}
J. Math. Mech.
{\bf 16} (1966), 327--338.

\end{thebibliography}
\end{document}